\documentclass[preprint,12pt]{elsarticle}

\usepackage{amssymb}
\usepackage{amsthm}
\usepackage{amsmath}
\usepackage{epic}
\usepackage{CJK}
\usepackage{setspace}
\newtheorem{thm}{Theorem}[section]
\newtheorem{cor}[thm]{Corollary}
\newtheorem{lem}[thm]{Lemma}

\journal{~}

\begin{document}
\begin{spacing}{1.15}
\begin{frontmatter}
\title{\textbf{Principal eigenvector and spectral radius of uniform hypergraphs}}

\author[label1]{Haifeng Li}
\author[label2]{Jiang Zhou}
\author[label1,label2]{Changjiang Bu}\ead{buchangjiang@hrbeu.edu.cn}


\address{
\address[label1]{College of Automation, Harbin Engineering University, Harbin 150001, PR China}
\address[label2]{College of Science, Harbin Engineering University, Harbin 150001, PR China}

}

\begin{abstract}
In this paper, we give some bounds for principal eigenvector and spectral radius of connected uniform hypergraphs in terms of vertex degrees, the diameter, and the number of vertices and edges.
\end{abstract}

\end{frontmatter}

\section{Introduction}
For a positive integer $n$, let $\left[ n \right] = \left\{ {1,2, \ldots ,n} \right\}$. Let $\mathbb{C}^{\left[ {m,n} \right]}$ be the set of order $m$ dimension $n$ tensors over the complex field $\mathbb{C}$. For $\mathcal{A} = \left( {a_{i_1 i_2  \cdots i_m } } \right) \in\mathbb{C}^{\left[ {m,n} \right]}$, if all the entries $a_{i_1 i_2  \cdots i_m }  \ge 0(or~  a_{i_1 i_2  \cdots i_m } >0)$ of $\mathcal{A}$, then $\mathcal{A}$ is called nonnegative (or positive) tensor. When $m = 2$, $\mathcal{A}$ is a $n \times n$ matrix. Let $\mathcal{I} = \left( {\delta _{i_1 i_2  \cdots i_m } } \right) \in \mathbb{C}^{\left[ {m,n} \right]}$ be the unit tensor, where $\delta _{i_1 i_2  \cdots i_m }$ is $Kronecker$ function.

 In 2005, Qi \cite{RA2005} and Lim \cite{Chang} defined the eigenvalue of tensors, respectively. For $\mathcal{A} = \left( {a_{i_1 i_2  \cdots i_m } } \right) \in \mathbb{C}^{\left[ {m,n} \right]}$ and $x = \left( {x_1 ,x_2 , \ldots ,x_n } \right)^\mathrm{T}  \in \mathbb{C}^n $, $\mathcal{A}x^{m - 1}$ is a dimension $n$ vector whose the $i$-th component is
\[
\left( {\mathcal{A}x^{m - 1} } \right)_i  = \sum\limits_{i_2 ,i_3 , \ldots ,i_m  = 1}^n {a_{ii_2  \cdots i_m } x_{i_2 } x_{i_3 }  \cdots x_{i_m } }.
\]
If there exists a number $\lambda  \in \mathbb{C}$, a nonzero vector $x = \left( {x_1 , \ldots ,x_n } \right)^\mathrm{T } \in \mathbb{C}^n$ such that
\[\mathcal{A}x^{m - 1}  = \lambda x^{\left[ {m - 1} \right]},\]
then $\lambda$ is called an eigenvalue of $\mathcal{A}$, $x$ is an eigenvector of $\mathcal{A}$ corresponding to the eigenvalue $\lambda$, where $x^{\left[ {m - 1} \right]}  = \left( {x_1^{m - 1} ,x_2^{m - 1} , \ldots ,x_n^{m - 1} } \right)^{\rm T}$. Let $\sigma \left( \mathcal{A} \right)$ denote the set of all eigenvalues of $\mathcal{A}$, the spectral radius $\rho \left( \mathcal{A} \right) = \mathop {\max } \left\{ {\left. {\left| {\lambda  } \right|} \right|\lambda  \in \sigma \left( \mathcal{A} \right)} \right\}$.

Chang et al. \cite{2008K.C. Chang}, Yang  et al. \cite{yang, YangaXiv2011}, Friedland et al. \cite{S.Friedland} gave Perron-Frobenius theorem of nonnegative tensors. Let $\mathcal{A} = \left( {a_{i_1 i_2  \cdots i_m } } \right)$ be an order $m$ dimension $n$ nonnegative tensor, if for any nonempty proper index subset $\alpha  \subset \left\{ {1, \ldots ,n} \right\}$, there is at least an entry
\[
a_{i_1  \cdots i_m }  > 0,~~~ where ~ i_1  \in \alpha~~ and ~at ~least ~an ~ \ i_j   \notin \alpha , ~~ j = 2 , \ldots ,m,
\]
then $\mathcal{A}$ is called nonnegative weakly irreducible tensor (see\cite{YangaXiv2011}).

Let ${V\left( G \right)} = \left\{ {\left. {1,2, \ldots ,n} \right\}} \right.$ and $E\left( G \right) = \left\{ {\left.{e_1,e_2, \ldots ,e_m} \right\}} \right.$ denote the vertex set and edge set of a hypergraph $G$, respectively. If each edge of $G$ contains exactly $k$ distinct vertices, then $G$ is called a $k$-uniform hypergraph. In particular, $2$-uniform hypergraphs are exactly the ordinary graphs. For a connected $k$-uniform hypergraph $G$, $e_i$ denotes an edge that contains vertex $i$, the degree of a vertex $i$ of $G$ is denoted by $d_i$, $\Delta = \max \{d_i\}$, $\delta = \min \{d_i\}$, for $i=1,\ldots ,n$. If all vertices of $G$ have the same degree, then $G$ is regular. A path $P$ of a $k$-uniform hypergraph is defined to be an alternating sequence of vertices and edges $P=v_{0}e_{1}v_{1}e_{2} \cdots v_{l-1}e_{l}v_{l}$, where $v_0,\ldots,v_{l}$ are distinct vertices of $G$, $e_1,\ldots,e_l$ are distinct edges of $G$ and $v_{i - 1} ,v_i  \in e_i $, for $i = 1, \ldots ,l$. The number of edges in $P$ is called the length of $P$. If there exists a path starting at $u$ and terminating at $v$ for all $u,v \in V(G)$, then $G$ is connected. Let $u$, $v$ be two distinct vertices of $G$, the distance between $u$ and $v$ is defined to be the length of the shortest path connecting them, denoted by $d(u,v)$. The diameter of a connected $k$-uniform hypergraph $G$ is the maximum distance among all vertices of $G$, denoted by $D$. In 2012,  Cooper and Dutle \cite{9} gave the concept of adjacency tensor of a $k$-uniform hypergraph. The adjacency tensor of a $k$-uniform hypergraph $G$ denoted by $\mathcal{A}_G$, is an order $k$ dimension $n$ nonnegative symmetric tensor with entries
\[a_{i_1 i_2  \cdots i_k }  = \left\{ \begin{gathered}
  \frac{1}
{{\left( {k - 1} \right)!}},{\kern 1pt} {\kern 1pt} {\kern 1pt} {\kern 1pt} {\kern 1pt} {\kern 1pt} {\kern 1pt} {\kern 1pt} {\kern 1pt} {\kern 1pt} {\kern 1pt} {\kern 1pt} {\kern 1pt} {\kern 1pt} {\kern 1pt} {\kern 1pt} {\kern 1pt} {\kern 1pt} {\kern 1pt} {\kern 1pt} {\kern 1pt} {\kern 1pt} {\kern 1pt} {\kern 1pt} {\kern 1pt} {\kern 1pt} {\kern 1pt} {\kern 1pt} {\kern 1pt} if\left\{ {\left. {i_1 ,i_2 , \ldots ,i_k } \right\}} \right. \in E\left( G \right), \hfill \\
  0,{\kern 1pt} {\kern 1pt} {\kern 1pt} {\kern 1pt} {\kern 1pt} {\kern 1pt} {\kern 1pt} {\kern 1pt} {\kern 1pt} {\kern 1pt} {\kern 1pt} {\kern 1pt} {\kern 1pt} {\kern 1pt} {\kern 1pt} {\kern 1pt} {\kern 1pt} {\kern 1pt} {\kern 1pt} {\kern 1pt} {\kern 1pt} {\kern 1pt} {\kern 1pt} {\kern 1pt} {\kern 1pt} {\kern 1pt} {\kern 1pt} {\kern 1pt} {\kern 1pt} {\kern 1pt} {\kern 1pt} {\kern 1pt} {\kern 1pt} {\kern 1pt} {\kern 1pt} {\kern 1pt} {\kern 1pt} {\kern 1pt} {\kern 1pt} {\kern 1pt} {\kern 1pt} {\kern 1pt} {\kern 1pt} {\kern 1pt} {\kern 1pt} {\kern 1pt} {\kern 1pt} {\kern 1pt} {\kern 1pt} {\kern 1pt} {\kern 1pt} {\kern 1pt} {\kern 1pt} {\kern 1pt} {\kern 1pt} {\kern 1pt} {\kern 1pt} {\kern 1pt} {\kern 1pt} {\kern 1pt} {\kern 1pt} {\kern 1pt} {\kern 1pt} {\kern 1pt} {\kern 1pt} {\kern 1pt} {\kern 1pt} {\text{otherwise}}. \hfill \\
\end{gathered}  \right.\]
Eigenvalues of $\mathcal{A}_G$ are called eigenvalues of $G$, the spectral radius of $\mathcal{A}_G$ is called the spectral radius of $G$, denoted by $\rho(G)$.

Let $G$ be a connected $k$-uniform hypergraph, then the adjacency tensor $\mathcal{\mathcal{A}}_G$ of $G$ is nonnegative weakly irreducible (see\cite{2014T. Zhang}), by Perron-Frobenius theorem of nonnegative tensors (see\cite{YangaXiv2011}), $\rho \left( G \right)$ is an eigenvalue of $\mathcal{\mathcal{A}}_G$, there exists a unique positive eigenvector $x = \left( {x_1 , \ldots ,x_n } \right)^\mathrm{T }$ corresponding to $\rho \left( G \right)$ and $\sum\limits_{i = 1}^n {x_i^k }  = 1$, $x$ is called and the principal eigenvector of $\mathcal{\mathcal{A}}_G$, the maximum and minimum entries of $x$ are denoted by $x_{max}$ and $x_{min}$, respectively. $\gamma = \frac{x_{max}}{x_{min}}$ is called the principal ratio of $\mathcal{\mathcal{A}}_G$ (see\cite{2007zhubilv}). In this paper, the principal eigenvector and the principal ratio of $\mathcal{\mathcal{A}}_G$ are called the principal eigenvector and the principal ratio of $G$. Let $\rho(G)$ be the spectral radius of a $k$-uniform hypergraph $G$ with eigenvector $x = \left( {x_1 , \ldots ,x_n } \right)^\mathrm{T }$. Since $\mathcal{A}_G x^{k-1}= \rho(G)x^{[k-1]}$, we know that $cx$ is also an eigenvector of $\rho(G)$ for any nonzero constant $c$. When $\sum\limits_{i = 1}^n {|x_i|^{k} }  = 1$, let $x^e=x_{i_1}x_{i_2}\cdots x_{i_k}$, $\{i_1,i_2,\ldots,i_k \}= e$ (see\cite{9, hypertree}), we have
\begin{eqnarray*}
\rho(G) & =&\frac{ x^{\rm{T}} (\mathcal{A}_G x^{k - 1} )}{x^{\rm{T}} (\mathcal{I} x^{k - 1} )}=x^{\rm{T}} (\mathcal{A}_G x^{k - 1} )=\sum\limits_{i_1, \ldots , i_k =1  }^{n} a_{i_1 \cdots  i_k} x_{i_1}  \cdots x_{i_k}\\
 &=& \sum\limits_{e \in E(G)} k! \frac{1}{(k-1)!} {x^e } = k\sum\limits_{e \in E(G)} {x^e }.
\end{eqnarray*}

In spectral graph theory, there are some work concern relations among the spectral radius, principal eigenvector and graph parameters \cite{2007zhubilv, 2005.Zhang.X.D}. The interest of this paper is to consider similar problems in spectral hypergraph theory. This paper is organized as follows. In Section 2, we give some bounds for the principal ratio and the maximum and minimum entries of principal eigenvector of connected uniform hypergraphs. In Section 3, we show some bounds for the spectral radius of connected uniform hypergraphs via degrees of vertices, the principal ratio and diameter.

\section{The principal eigenvector of hypergraphs}

Let $G$ be a connected $k$-uniform hypergraph, $G$ is regular if and only if  $\gamma=1$. Thus, $\gamma$ is an index which measure the irregularity of $G$. In 2005, Zhang \cite{2005.Zhang.X.D} gave some bounds of the principal ratio of irregular graph $G$, these results were used to obtain a bound of the spectral radius of $G$.

Let $G$ be a connected uniform hypergraph with maximum degree $\Delta$, minimum degree $\delta$.  We give the lower  bound for the principal ratio $\gamma$ of $G$, which extend the result of  Zhang [7, Theorem 2.3] to hypergraphs.
\begin{thm}\label{dingli3.1}
Let $G$ be a connected $k$-uniform hypergraph, then
\begin{align}\label{2.1}
\gamma  \ge \left( {\frac{\Delta }{\delta }} \right)^{\frac{1}{{2\left( {k - 1} \right)}}}.\tag{2.1}
\end{align}
If equality in $(2.1)$ holds, then $\rho \left( G \right) = \sqrt {\Delta \delta } .$
\end{thm}
\begin{proof}
Let $G$ be a connected $k$-uniform hypergraph , $\mathcal{A}_G$ be the adjacency tensor of $G$, $x = \left( {x_1 , \ldots ,x_n } \right)^\mathrm{T}$ be the principal eigenvector of $G$. Suppose that $d_p  = \Delta$, $d_q  = \delta$, $(p,q \in V(G))$, since $\mathcal{A}_Gx^{k - 1}  = \rho \left( G \right)x^{\left[ {k - 1} \right]}$, we have
\begin{align}\label{2.2}
\rho \left( G \right)x_p^{k - 1}  = \sum\limits_{e_p  \in E\left( G \right)} {x^{e_p\backslash \left\{ p \right\}} }  \ge \Delta x_{\min }^{k - 1} .\tag{2.2}
\end{align}
\begin{align}\label{2.3}
\rho \left( G \right)x_q^{k - 1}  = \sum\limits_{e_q  \in E\left( G \right)} {x^{e_q\backslash \left\{ q \right\}} }  \le \delta x_{\max }^{k - 1} .\tag{2.3}
\end{align}
where $x^{e_i \backslash \left\{ i \right\}}  = x_{i_2 } x_{i_3 }  \cdots x_{i_k }$, $\{i, i_2,\ldots,i_k\} = e_i$.
\\By $(\ref{2.2})$ and $(\ref{2.3})$, we have
\begin{align}\label{2.4}
\frac{\Delta }{{\rho \left( G \right)}}\frac{{\rho \left( G \right)}}{\delta } \le \left( {\frac{{x_p }}{{x_{\min } }}} \right)^{k - 1} \left( {\frac{{x_{\max } }}{{x_q }}} \right)^{k - 1}  \le \left( {\frac{{x_{\max } }}{{x_{\min } }}} \right)^{2\left( {k - 1} \right)} {\rm{ = }}\gamma ^{^{2\left( {k - 1} \right)} }.\tag{2.4}
\end{align}
i.e.\[
\gamma  \ge \left( {\frac{\Delta }{\delta }} \right)^{\frac{1}{{2\left( {k - 1} \right)}}}.
\]
Since equality in $(\ref{2.1})$ holds, three equalities in (\ref{2.2}), (\ref{2.3}) and (\ref{2.4}) hold. If equality in (\ref{2.4}) holds, we have
\begin{align}\label{2.5}
x_p=x_{\max}, ~~x_q=x_{\min}.\tag{2.5}
\end{align}
When equalities in (\ref{2.2}) and (\ref{2.3}) hold, by (\ref{2.5}), we obtain
\[\rho \left( G \right) = \sqrt {\Delta \delta }.\]
\end{proof}

Applying the bound of the principal ratio $\gamma$, we obtained the result as follows.
\begin{thm}\label{zuida,zuixiao}
Let $G$ be a connected $k$-uniform hypergraph, $x = \left( {x_1 , \ldots ,x_n } \right)^\mathrm{T}$ be the principal eigenvector of $G$, then \\
$(1)$$x_{\max }  \ge \left[ {\left( {\frac{\delta }{\Delta }} \right)^{\frac{k}{{2\left( {k - 1} \right)}}}  + n - 1} \right]^{ - \left( {\frac{1}{k}} \right)}$;\\
$(2)$$x_{\min }  \le \left[ {\left( {\frac{\Delta }{\delta }} \right)^{\frac{k}{{2\left( {k - 1} \right)}}}  + n - 1} \right]^{ - \left( {\frac{1}{k}} \right)}$.
\end{thm}
\begin{proof}
Let $G$ be a connected $k$-uniform hypergraph, $x = \left( {x_1 , \ldots ,x_n } \right)^\mathrm{T}$ be the principal eigenvector of $G$, then
\[
1{\rm{ = }}\sum\limits_{i = 1}^n {x_i^k }  \le x_{\min }^k  + \left( {n - 1} \right)x_{\max }^k.
\]
Let $\gamma$ be the principal ratio of $G$, we obtain
\[
x_{\max }^{-k} \leq {\gamma ^{ - k}  + n - 1} ,
\]
\begin{align}\label{2.6}
x_{\max }^k  \ge \left( {\gamma ^{ - k}  + n - 1} \right)^{ - 1} .\tag{2.6}
\end{align}
 By Theorem \ref{dingli3.1}, we know that $\gamma  \ge \left( {\frac{\Delta }{\delta }} \right)^{\frac{1}{{2\left( {k - 1} \right)}}}$, so
\[
x_{\max }^k    \ge \left[ {\left( {\frac{\delta }{\Delta }} \right)^{\frac{k}{{2\left( {k - 1} \right)}}}  + n - 1} \right]^{ - 1},
\]
i.e.
\begin{align}\label{2.7}
x_{\max }  \ge \left[ {\left( {\frac{\delta }{\Delta }} \right)^{\frac{k}{{2\left( {k - 1} \right)}}}  + n - 1} \right]^{ - \left( {\frac{1}{k}} \right)}.\tag{2.7}
\end{align}
Since
\[
1{\rm{ = }}\sum\limits_{i = 1}^n {x_i^k }  \ge \left( {n - 1} \right)x_{\min }^k  +x_{\max }^k,
\]
\[
x_{\min }^{-k } \ge  {\gamma ^k  + n - 1} ,
\]
\[
x_{\min }^k  \le \left( {\gamma ^k  + n - 1} \right)^{ - 1}.
\]
By Theorem \ref{dingli3.1}, we know that $\gamma  \ge \left( {\frac{\Delta }{\delta }} \right)^{\frac{1}{{2\left( {k - 1} \right)}}}$, so
\[
x_{\min }^k  \le  \left[ {\left( {\frac{\Delta }{\delta }} \right)^{\frac{k}{{2\left( {k - 1} \right)}}}  + n - 1} \right]^{ - 1},
\]
i.e.
\begin{align*}
x_{\min }  \le \left[ {\left( {\frac{\Delta }{\delta }} \right)^{\frac{k}{{2\left( {k - 1} \right)}}}  + n - 1} \right]^{ - \left( {\frac{1}{k}} \right)}.
\end{align*}
\end{proof}

\begin{thm}\label{zuida}
Let $G$ be a connected $k$-uniform hypergraph with $n$ vertices and $m$ edges, $x = \left( {x_1 , \ldots ,x_n } \right)^\mathrm{T}$ be the principal eigenvector of $G$, then
\begin{equation*}
x_{\max }  \ge  \left( {\frac{{\rho \left( {G } \right)}}{{km}}} \right)^{\frac{1}{k}},
\end{equation*}
with equality if and only if $G$ is regular.
\end{thm}
\begin{proof}
Let $G$ be a connected $k$-uniform hypergraph with $n$ vertices and $m$ edges, $\mathcal{A}_G  = \left( {a_{i i_2  \cdots i_k } } \right)$ be the adjacency tensor of $G$, $x = \left( {x_1 , \ldots ,x_n } \right)^\mathrm{T}$ be the principal eigenvector of $G$, then
\[
\rho \left( G \right) = x^T \left( {\mathcal{A}_G x^{k - 1} } \right) = k\sum\limits_{e \in E\left( G \right)} {x^e }  \le kmx_{\max }^k.\tag{2.8}
\]
\[
x_{\max }  \ge  \left( {\frac{{\rho \left( {G } \right)}}{{km}}} \right)^{\frac{1}{k}}.\tag{2.9}
\]
Clearly, equality in (2.9) holds if and only if equality in (2.8) holds, i.e. $x_1  = x_2  =  \cdots  = x_n$, therefore $G$ is regular.
\end{proof}

\section{The spectral radius of hypergraphs}
Let $G$ be a connected uniform hypergraph with maximum degree $\Delta$, minimum degree $\delta$.  We obtain some bounds for the spectral radius of $G$ via degrees of vertices, the principal ratio and diameter.

We give some auxiliary lemmas which will be used in the sequel.

\begin{lem}\rm\cite{ChangAn.Cooper}\label{budengshi}
Let $y_1,\ldots,y_n$ be nonnegative real numbers $(n \geq 2)$, then
\begin{eqnarray*}
\frac{y_1+\cdots+y_n}{n}-(y_1\cdots y_n)^{\frac{1}{n}}\geqslant\frac{1}{n(n-1)}\sum_{1\leqslant i<j\leqslant n}(\sqrt{y_i}-\sqrt{y_j})^2.
\end{eqnarray*}
\end{lem}
\begin{lem}\label{zhoulem2.2}
Let $a,b,y_1,y_2$ be positive numbers. Then $a(y_1-y_2)^2+by_2^2\geqslant\frac{ab}{a+b}y_1^2$.
\end{lem}
\begin{proof}
By computation, we have
\begin{eqnarray*}
a(y_1-y_2)^2+by_2^2=(a+b)(y_2-\frac{ay_1}{a+b})^2+\frac{ab}{a+b}y_1^2\geqslant\frac{ab}{a+b}y_1^2.
\end{eqnarray*}
\end{proof}

\begin{thm}\label{definition2.6}
Let $G$ be a connected $k$-uniform hypergraph, then
\[
\frac{\Delta }{{\gamma ^{k - 1} }} \le \rho \left( G \right) \le \gamma ^{k - 1} \delta .
\]
\end{thm}
\begin{proof}
Let $x = \left( {x_1 , \ldots ,x_n } \right)^\mathrm{T}$ be the principal eigenvector of $G$, for all $i \in V(G)$, we have
\[
\rho \left( G \right)x_i^{k - 1}  = \sum\limits_{e_i  \in E(G)} {x^{e_i \backslash \left\{ i \right\}} }  \ge d_i x_{\min }^{k - 1}  > 0,
\]
Suppose that $d_\mu  =\delta $, $\mu \in V(G)$, we obtain
\begin{equation*}
\rho \left(G \right) =\sum\limits_{e_\mu  \in E(G)} \frac{{x^{e_\mu \backslash \left\{ \mu \right\}} }}{{x_\mu^{k - 1} }}\leq \gamma ^{k - 1} \delta.
\end{equation*}
Similarly, we have
\begin{equation*}
\rho \left( G\right) \ge \frac{\Delta }{{\gamma ^{k - 1} }}.
\end{equation*}
Thus,
\[
\frac{\Delta }{{\gamma ^{k - 1} }} \le \rho \left( G \right) \le \gamma ^{k - 1} \delta .
\]
\end{proof}

\begin{thm}\label{zhoudingli}
Let $G$ be an irregular connected $k$-uniform hypergraph with $n$ vertices and $m$ edges, then
\begin{equation*}
\rho (G) < \frac{{km\Delta }}{{km + (n\Delta  - km){\gamma }^{-k}  + \frac{k}{{2(k - 1)D}}\left[ {1 - {\gamma }^{-\frac{k}{2}} } \right]^2 }},
\end{equation*}
where $D$ is the diameter of $G$.
\end{thm}
\begin{proof}
Let $G$ be an irregular connected $k$-uniform hypergraph with $n$ vertices and $m$ edges, $x = \left( {x_1 , \ldots ,x_n } \right)^\mathrm{T}$ is the principal eigenvector of $G$, then
\begin{eqnarray*}
\Delta-\rho(G)&=&\Delta\sum_{i=1}^nx_i^k-k\sum_{e\in E(G)}x^e\\
&=&\sum_{i=1}^n(\Delta-d_i)x_i^k+\sum_{i=1}^nd_ix_i^k-k\sum_{e\in E(G)}x^e\\
&=&\sum_{i=1}^n(\Delta-d_i)x_i^k+\sum_{\{i_1\cdots i_k \}= e \in E(G)}(x_{i_1}^k+\cdots+x_{i_k}^k-kx^e).
\end{eqnarray*}
Let $x_u  = \mathop {\max }\limits_i x_i $, $x_v  = \mathop {\min }\limits_i x_i  $, since $x> 0 $, $G$ be an irregular connected $k$-uniform hypergraph, by Lemma \ref{budengshi}, it yields that
\begin{equation}\label{3.1}
\Delta-\rho(G)>(n\Delta-km)x_v^k+\frac{1}{k-1}\sum_{i,j\in e\in E(G)}(x_i^{\frac{k}{2}}-x_j^{\frac{k}{2}})^2.\tag{3.1}
\end{equation}
Let $P=v_{0}e_{1}v_{1}e_{2} \cdots v_{l-1}e_{l}v_{l}$ be the shortest path from vertex $u$ to vertex $v$, where $u=v_{0}$, $v=v_{l}$，$v_{i - 1} ,v_i  \in e_i ,for~i = 1, \ldots ,l$, we have
\begin{eqnarray*}
\sum_{i,j\in e\in E(P)}(x_i^{\frac{k}{2}}-x_j^{\frac{k}{2}})^2
&\geq& \sum_{i=0}^{l-1}(x_{v_i}^{\frac{k}{2}}-x_{v_{i+1}}^{\frac{k}{2}})^2\\
&+&\sum_{i=0}^{l-1}\sum_{u_i\in e_i\backslash\{v_{i-1},v_i\}}[(x_{v_i}^{\frac{k}{2}}-x_{u_{i+1}}^{\frac{k}{2}})^2+(x_{u_{i+1}}^{ \frac{k}{2}}-x_{v_{i+1}}^{\frac{k}{2}})^2]\\
&\geq& \sum_{i=0}^{l-1}(x_{v_i}^{\frac{k}{2}}-x_{v_{i+1}}^{\frac{k}{2}})^2+\frac{1}{2}\left(\sum_{i=0}^{ l-1}\sum_{u_i\in e_i\backslash\{v_{i-1},v_i\}}(x_{v_i}^{\frac{k}{2}}-x_{v_{i+1}}^{\frac{k}{2}})^2\right)\\
&=&\sum_{i=0}^{l-1}(x_{v_i}^{\frac{k}{2}}-x_{v_{i+1}}^{\frac{k}{2}})^2+\frac{k-2}{2}\left( \sum_{i=0}^{l-1}(x_{v_i}^{\frac{k}{2}}-x_{v_{i+1}}^{\frac{k}{2}})^2\right).
\end{eqnarray*}
By Cauchy-Schwarz inequality, we obtain
\begin{eqnarray*}
\sum_{i,j\in e\in E(P)}(x_i^{\frac{k}{2}}-x_j^{\frac{k}{2}})^2
&\geq&\frac{1}{l}\left(\sum_{i=0}^{l-1}(x_{v_i}^{\frac{k}{2}}-x_{v_{i+1}}^{\frac{k}{2}})\right)^2+ \frac{k-2}{2l}\left(\sum_{i=0}^{l-1}(x_{v_i}^{\frac{k}{2}}- x_{v_{i+1}}^{ \frac{k}{2}})\right)^2\\
&=&\frac{1}{l}(x_u^{\frac{k}{2}}-x_v^{\frac{k}{2}})^2+\frac{k-2}{2l}(x_u^{\frac{k}{2}}-x_v^{ \frac{k}{2}})^2\\
&=&\frac{k}{2l}(x_u^{\frac{k}{2}}-x_v^{\frac{k}{2}})^2.
\end{eqnarray*}
Let $D$ is the diameter of $G$, since $l=d(u,v)\leq D$, so
\begin{equation}\label{3.2}
\sum_{i,j\in e\in E(P)}(x_i^{\frac{k}{2}}-x_j^{\frac{k}{2}})^2 \geq \frac{k}{2D}(x_u^{\frac{k}{2}}-x_v^{\frac{k}{2}})^2.\tag{3.2}
\end{equation}
By (\ref{3.1}) and (\ref{3.2}), it yields that
\begin{align}\label{3.3}
\Delta-\rho(G)>(n\Delta-km)x_v^k+\frac{k}{2(k-1)D}(x_u^{\frac{k}{2}}-x_v^{\frac{k}{2}})^2.\tag{3.3}
\end{align}
Let $\gamma$ be the principal ratio of $G$, we have
\begin{align}\label{3.4}
\frac{{\Delta  - \rho (G)}}{{x_u^k }} > (n\Delta  - km){\gamma }^{-k}  + \frac{k}{{2(k - 1)D}}\left[ {1 - {\gamma }^{-\frac{k}{2}} } \right]^2.\tag{3.4}
\end{align}
It follows from Theorem \ref{zuida} that $x_{u} ^{k} \ge   {\frac{{\rho \left( {G } \right)}}{{km}}} $, so
\[
\frac{{\left( {\Delta  - \rho (G)} \right)km}}{{\rho (G)}} > \frac{{\Delta  - \rho (G)}}{{x_\mu ^k }} > (n\Delta  - km)\gamma ^{ - k}  + \frac{k}{{2(k - 1)D}}\left[ {1 - \gamma ^{ - \frac{k}{2}} } \right]^2,
\]

\[
\rho (G) < \frac{{km\Delta }}{{km + (n\Delta  - km){\gamma }^{-k}  + \frac{k}{{2(k - 1)D}}\left[ {1 - {\gamma }^{-\frac{k}{2}} } \right]^2 }}.
\]
\end{proof}

\begin{thm}
Let $G$ be an irregular connected $k$-uniform hypergraph with $n$ vertices and $m$ edges, $D$ is the diameter of $G$, then
\[
 \rho (G) < \Delta- \frac{{2(k - 1)D(n\Delta  - km)\gamma ^{ - k}  + k\left( {1 - \gamma ^{ - \frac{k}{2}} } \right)^2 }}{{2\left( {\gamma ^{ - k}  + n - 1} \right)(k - 1)D}}.
\]
\end{thm}
\begin{proof}
Let $x = \left( {x_1 , \ldots ,x_n } \right)^\mathrm{T }$, $\gamma$ be the principal eigenvector and the principal ratio of $G$, respectively, $x_u  = \mathop {\max }\limits_i x_i $, by (\ref{3.4}), we know
\begin{eqnarray*}
\frac{{\Delta  - \rho (G)}}{{x_u^k }} > (n\Delta  - km){\gamma }^{-k}  + \frac{k}{{2(k - 1)D}}\left[ {1 - {\gamma }^{-\frac{k}{2}} } \right]^2.
\end{eqnarray*}
By (\ref{2.6}), we have
\begin{eqnarray*}
x_{u }^k  \ge \left( {\gamma ^{ - k}  + n - 1} \right)^{ - 1} .
\end{eqnarray*}
Thus
\begin{eqnarray*}
{{\Delta  - \rho (G)}}&>& \frac{(n\Delta  - km){\gamma }^{-k}  + \frac{k}{{2(k - 1)D}}\left[ {1 - {\gamma }^{-\frac{k}{2}} } \right]^2}{{\gamma ^{ - k}  + n - 1}}\\
&= &\frac{{2(k - 1)D(n\Delta  - km)\gamma ^{ - k}  + k\left( {1 - \gamma ^{ - \frac{k}{2}} } \right)^2 }}{{2\left( {\gamma ^{ - k}  + n - 1} \right)(k - 1)D}}.
\end{eqnarray*}
\[
 \rho (G) < \Delta- \frac{{2(k - 1)D(n\Delta  - km)\gamma ^{ - k}  + k\left( {1 - \gamma ^{ - \frac{k}{2}} } \right)^2 }}{{2\left( {\gamma ^{ - k}  + n - 1} \right)(k - 1)D}}.
\]
\end{proof}

\begin{thm}\label{jieguo1}
Let $G$ be an irregular connected $k$-uniform hypergraph with $n$ vertices and $m$ edges, $D$ is the diameter of $G$, then
\[
\rho(G)<\Delta - \frac{k(n\Delta-km)} {\left[{2(k-1)D(n\Delta-km)+k} \right]\left[{\left( {\frac{\delta }{\Delta }} \right)^{\frac{k}{{2\left( {k - 1} \right)}}}  + n - 1}\right]}.
\]
\end{thm}
\begin{proof}
Let $G$ be an irregular connected $k$-uniform hypergraph with $n$ vertices and $m$ edges, Let $x = \left( {x_1 , \ldots ,x_n } \right)^\mathrm{T }$, $\gamma$ be the principal eigenvector and the principal ratio of $G$, respectively, $x_u  = \mathop {\max }\limits_i x_i $, $x_v  = \mathop {\min }\limits_i x_i $, by (\ref{3.3}), we know
\begin{eqnarray*}
\Delta-\rho(G)>(n\Delta-km)x_v^k+\frac{k}{2(k-1)D}(x_u^{\frac{k}{2}}-x_v^{\frac{k}{2}})^2.
\end{eqnarray*}
By Lemma \ref{zhoulem2.2}, we get
\begin{align}\label{zhoujieguo}
\Delta-\rho(G)>\frac{k(n\Delta-km)}{2(k-1)D(n\Delta-km)+k}x_u^k.\tag{3.5}
\end{align}
By Theorem \ref{zuida,zuixiao}, we have
\[
x_{u}  \ge \left[ {\left( {\frac{\delta }{\Delta }} \right)^{\frac{k}{{2\left( {k - 1} \right)}}}  + n - 1} \right]^{ - \left( {\frac{1}{k}} \right)}.
\]
Thus,
\begin{eqnarray*}
\Delta-\rho(G)>\frac{k(n\Delta-km)} {\left[{2(k-1)D(n\Delta-km)+k} \right]\left[{\left( {\frac{\delta }{\Delta }} \right)^{\frac{k}{{2\left( {k - 1} \right)}}}  + n - 1}\right]}.
\end{eqnarray*}
i.e.
\[
\rho(G)<\Delta - \frac{k(n\Delta-km)} {\left[{2(k-1)D(n\Delta-km)+k} \right]\left[{\left( {\frac{\delta }{\Delta }} \right)^{\frac{k}{{2\left( {k - 1} \right)}}}  + n - 1}\right]}.
\]
\end{proof}

\begin{thm}\label{jieguo2}
Let $G$ be an irregular connected $k$-uniform hypergraph with $n$ vertices and $m$ edges, then
\[
\rho(G)<\frac{2m\Delta(k-1)D(n\Delta-km)+km\Delta}{ 2m(k-1)D(n\Delta-km)+n\Delta},
\]
where $D$ is the diameter of $G$.
\end{thm}
\begin{proof}
Let $x = \left( {x_1 , \ldots ,x_n } \right)^\mathrm{T }$, $\gamma$ be the principal eigenvector and the principal ratio of $G$, respectively, $x_u  = \mathop {\max }\limits_i x_i $, by (\ref{zhoujieguo}), we know
\[
\Delta-\rho(G)>\frac{k(n\Delta-km)}{2(k-1)D(n\Delta-km)+k}x_u^k.
\]
By Theorem \ref{zuida}, we have
\begin{equation*}
x_{u}  \ge  \left( {\frac{{\rho \left( {G } \right)}}{{km}}} \right)^{\frac{1}{k}},
\end{equation*}
Thus
\[
km(\Delta-\rho(G))>\frac{k(n\Delta-km)\rho(G)}{2(k-1)D(n\Delta-km)+k}.
\]
\[
\rho(G)<\frac{km\Delta}{km+\frac{k(n\Delta-km)}{2(k-1)D(n\Delta-km)+k}}=\frac{2m\Delta(k-1)D(n\Delta-km)+km\Delta}{ 2m(k-1)D(n\Delta-km)+n\Delta}.
\]
\end{proof}

Let $G$ be an irregular connected $k$-uniform hypergraph with $n$ vertices and $m$ edges, when $k=2$, by Theorem \ref{jieguo1} and Theorem \ref{jieguo2}, we can obtain results as follows.
\begin{cor}\label{graphone}
Let $G$ be an irregular connected graph with $n$ vertices and $m$ edges, $D$ is the diameter of $G$, then
\[
\Delta - \rho(G)> \frac{n\Delta-2m} {\left[{D(n\Delta-2m)+1} \right]\left[ {\frac{\delta }{\Delta } + n - 1}\right]} >\frac{n\Delta-2m}{n(D(n\Delta-2m)+1)}.
\]
\end{cor}

\begin{cor}\label{graphtwo}
Let $G$ be an irregular connected graph with $n$ vertices and $m$ edges, $D$ is the diameter of $G$, then
\[
\Delta- \rho(G)> \Delta - \frac{2m\Delta D(n\Delta-2m)+2m\Delta}{ 2m D(n\Delta-2m)+n\Delta}>\frac{n\Delta-2m}{n(D(n\Delta-2m)+1)}.
\]
\end{cor}
\noindent
\textbf{Remark:} For a connected irregular graph $G$ with $n$ vertices and $m$ edges, Cioab\u{a}, Gregory and Nikiforov \cite{Cioaba} obtain the following bound
\begin{align}\label{3.6}
\Delta-\rho(G)>\frac{n\Delta-2m}{n(D(n\Delta-2m)+1)}.\tag{3.6}
\end{align}
Clearly, the results of Corollary\ref{graphone} and Corollary \ref{graphtwo} improve bound (\ref{3.6}).

\vspace{3mm}
\noindent
\textbf{Acknowledgements.}

\vspace{3mm}
This work is supported by the National Natural Science Foundation of China (No. 11371109), the Natural Science Foundation of the Heilongjiang Province (No. QC2014C001) and the Fundamental Research Funds for the Central Universities.

\end{spacing}

\vspace{3mm}
\noindent
\textbf{References}

\begin{thebibliography}{}
\bibitem{RA2005}L. Qi, Eigenvalues of a real supersymmetric tensor, J. Symbolic Comput. 40(2005) 1302-1324.
\bibitem{Chang}L.H. Lim, Singular values and eigenvalues of tensors: a variational approach, in: Proceedings 1st IEEE International Workshop on Computational Advances of Multiensor Adaptive Processing, (2005) 129-132.
\bibitem{yang}Y.Yang and Q. Yang, Further results for Perron-Frobenius theorem for nonnegative tensors, SIAM. J. Matrix Anal. Appl. 31(2010) 2517-2530.
\bibitem{YangaXiv2011}Y.Yang and Q. Yang, On some properties of nonnegative weakly tensors, aXiv:1111.0713.
\bibitem{2007zhubilv}S.M. Cioab$\breve{a}$ and D.A. Gregory, Principal eigenvectors of irregular graphs, Electronic Journal of Linear Algebra, 16(2007) 366-379.
\bibitem{9}J. Cooper, A. Dutle, Spectral of a uniform hypergraph, Linear Algebra Appl. 436 (2012) 3268-3292.
\bibitem{2005.Zhang.X.D}X.D. Zhang, Eigenvectors and eigenvalues of non-regular graphs, Linear Algebra Appl. 409(2005) 79-86.
\bibitem{2014T. Zhang}K. Pearson, T. Zhang, On spectral hypergraph theory of the adjacency tensor, Graphs and Combin. 30(2014) 1233-1248.
\bibitem{ChangAn.Cooper}A. Chang, J. Cooper, W. Li, Analytic connectivity of $k$-uniform hypergraphs, arXiv:1507.02763v1.
\bibitem{hypertree}H. Li, J.Y. Shao, L. Qi, The extremal spectral radii of $k$-uniform supertrees, J. Comb. Optim. DOI:10.1007/s10878-015-9896-4.
\bibitem{2008K.C. Chang}K.C. Chang, Kelly Pearson and Tan Zhang, Perron-Frobenius theorem for nonnegative tensors, Commun. Math. Sci. 6(2008) 507-520.
\bibitem{S.Friedland}S.Friedland, S.Gaubert and L.Han, Perron-Frobenius theorem for nonnegative multilinear forms and extensions, Linear Algebra Appl. 438(2013) 738-749.
\bibitem{Cioaba}S.M. Cioab\u{a}, D.A. Gregory, V. Nikiforov, Extreme eigenvalues of nonregular graphs, J. Combin. Theory, Ser. B 97 (2007) 483-486.
\end{thebibliography}
\end{document}